\documentclass[10pt,leqno]{amsart}
\usepackage{amsmath}
\usepackage{amsfonts}
\usepackage{amssymb}
\usepackage{mathtools}
\usepackage{graphicx}
\usepackage{color}
\usepackage{hyperref}
\allowdisplaybreaks[4]

\usepackage{xcolor}
\newtheorem{theorem}{Theorem}[section]
\newtheorem*{theorem*}{Theorem}
\newtheorem{lemma}[theorem]{Lemma}

\newtheorem{proposition}[theorem]{Proposition}

\theoremstyle{definition}
\newtheorem{definition}[theorem]{Definition}

\newtheorem{remark}{Remark}[section]

\newcommand{\R}{\mathbb{R}}

\def\Ric{\text{Ric}}
\def\lf{\left}
\def\ri{\right}

\def\a{\alpha}
\def\l{\lambda}

\def\R{\mathbb{R}}

\def\S{{\operatorname{Scal}}}

\def\vp{\varphi}

\def\Ric{\operatorname{Ric}}

\def\tr{\operatorname{tr}}

\numberwithin{equation}{section}
\makeatletter
\newcommand*\owedge{\mathpalette\@owedge\relax}
\newcommand*\@owedge[1]{%
  \mathbin{%
    \ooalign{%
      $#1\m@th\bigcirc$\cr
      \hidewidth$#1\m@th\wedge$\hidewidth\cr
    }%
  }%
}
\makeatother

\begin{document}

\title[Secondary Curvature Operator on Einstein Manifolds]{
Einstein manifolds of negative lower bounds on curvature operator of the second Kind}

\author[Cheng]{Haiqing Cheng}
\address{School of Mathematical Sciences, Soochow University, Suzhou, 215006, China}
\email{chq4523@163.com}

\author[Wang]{Kui Wang}\thanks{}
\address{School of Mathematical Sciences, Soochow University, Suzhou, 215006, China}
\email{kuiwang@suda.edu.cn}

\subjclass[2020]{53C20, 53C24, 53C25}

\keywords{Curvature operator of the second kind, Einstein manifolds, Sphere theorems}

\begin{abstract}
We demonstrate  that  $n$-dimension closed Einstein manifolds, whose smallest eigenvalue of the curvature operator of the second kind  of $\mathring{R}$ satisfies
$\lambda_1 \ge -\theta(n) \bar\lambda$,  are either flat  or round spheres, where $\bar \lambda$ is the average of the eigenvalues of $\mathring{R}$,  and $\theta(n)$ is  defined as in equation \eqref{1.2}. Our result improves  a celebrated result (Theorem \ref{thmp}) concerning  Einstein manifolds with nonnegative curvature operator of the second kind. 
\end{abstract}
\maketitle

\section{Introduction and main results}
In this paper, we explore the sphere theorems for Einstein manifolds within the context of
conditions imposed by the curvature operator of the second kind. The study of the curvature
operator of the second kind, as defined in Section \ref{sec 2.2} below, can be traced back to a
conjecture by Nishikawa \cite{Nishikawa86}, which asserts that a closed Riemannian manifold with positive (or nonnegative)  curvature
operator of the second kind is diffeomorphic to a spherical space form (or a Riemannian locally symmetric space). Recently, Cao-Gursky-Tran \cite{CGT23} and Li \cite{Li21} have  confirmed Nishikawa’s conjecture for the positive and nonnegative cases, respectively. For further
information and recent developments on this topic, we direct readers to the following
references:\cite{DF24,DFY24, FL24,Li22Kahler,Li22PAMS,Li22product,Li24,NPW22,NPWW22}.

 For Einstein manifolds, a well-known result in this aspect is the following theorem.
\begin{theorem}\label{thmp}
    Any closed Einstein manifold with nonnegative curvature operator of the second kind is either flat or a round sphere.  
\end{theorem}
There is a rich literature on Theorem \ref{thmp}.  Kashiwada \cite{Kashiwada93} proved that compact manifold $M$ with harmonic curvature and positive curvature operator of the second kind  is a space of constant curvature.
Cao-Gursky-Tran \cite{CGT23} demonstrated that compact Einstein manifolds  with 4-positive and 4-nonnegative curvature operator of the second kind are constant curvature spaces and locally symmetric spaces, respectively. Subsequently, Li \cite{Li22JGA} relaxed the assumption to $4\frac{1}{2}$-positive (respectively, nonnegative) curvature operator of the second kind. 
Nienhaus-Petersen-Wink \cite{NPW22} showed that $n$-dimensional compact Einstein manifolds with $ N'$-nonnegative ($ N'<\frac{3n(n+2)}{2(n+4)}$) curvature operators of the second kind are either rational homology spheres or ﬂat. 
Recently, Dai, Fu and Yang \cite{DFY24}  showed that a compact Einstein manifold of dimension $n\ge 4$ with nonnegative curvature operator of the second kind is flat or  a space form of positive curvature, see also  \cite{DF24}.

 Cao-Gursky-Tran \cite{CGT23}, Nienhaus-Petersen-Wink \cite{NPW22}, Li \cite{Li21} and Dai-Fu \cite{DF24} have proved many excellent results under various curvature conditions. Recently, Li \cite{Li24} introduced new lower bound conditions on the curvature operator of the second kind (see  \cite[Definition 1.2]{Li24})  and proved optimal sphere theorems. Based on this framework, we focus on Li's condition with  $\a=1$, namely \eqref{1.1}, which is weaker than nonnegative curvature operator of the second kind. To state our result, we  introduce some notation: $(M^n, g)$ is an $n$-dimensional Riemannian manifold,  $\mathring{R}$ represents the curvature operator of the second kind,  $N=(n+2)(n-1)/2$ is the dimension of $S_0^2(T_p M)$, $\l_1\le \l_2\le \cdots\le \l_N$ are the eigenvalues of  $\mathring{R}$, and $\bar{\l}=\sum_{j=1}^N \l_j/N$ is the average of $\l_j$. The main result of this paper is as follows:
\begin{theorem}\label{thm1}
Let $(M^n,g)$ be a closed Einstein manifold of dimension $n \ge 4$.
If
\begin{align}\label{1.1}
\l_1\ge -\theta(n) \bar{\l},
\end{align}
then $M$ is  either  flat  or a round sphere. Here 
$\theta(n)$ is a positive constant given by
\begin{align}\label{1.2}
\theta(n)=
\begin{cases}
17/7, &\quad n=4,\\
47/18, &\quad n=5,\\
\frac{2 (n^4-4 n^3-5 n^2+30 n-16)}{n^4+8 n^3-5 n^2-48 n+32}, &\quad n\ge 6.    
\end{cases}
 \end{align}
\end{theorem}
\begin{remark}
      Theorem \ref{thm1} essentially represents a gap theorem for  Einstein manifolds,  in the sense  that  if the curvature operator of the second kind satisfies \eqref{1.1}, then  the curvature operator of the second kind is nonnegative. 
\end{remark}

The crucial step in proving Theorem \ref{thm1} is solving an extremum problem introduced in Proposition  \ref{pro 3.4}. Upon examination of the proof, it becomes evident that for any $\theta$, strictly greater than $\theta(n)$,  Proposition  \ref{pro 3.4} does not hold.  Consequently, by employing our method, $\theta(n)$ as provided in equation \eqref{1.2} is deemed optimal. But it remains an interesting question that whether $\theta(n)$ is sharp or not for Theorem \ref{thm1}. Recently, the authors \cite{CW25}  obtained analogous results under a cone condition where the first two eigenvalues have a negative lower bound.

The subsequent sections of this paper are structured as follows. In Section \ref{sec 2}, we give the definition of the curvature operator of the second kind and derive some identities for the curvatures on Einstein manifolds.
Section \ref{sec 3} is devoted to the proof of Theorem \ref{thm1}.

\section{Preliminaries}\label{sec 2}

\subsection{Notation and Conventions}
In this section, we set up  some  notations and recall some basic facts about the curvature operator of the second kind, and refer to \cite{BK78, CV60,Kashiwada93,Li21, Nishikawa86, OT79} for more details.

Let $(M^n, g)$ be an $n$-dimensional manifold,  $V=T_p M$,  $\displaystyle \{e_i\}_{i=1}^n$ be an orthonormal basis of $V$ with respect to $g$, and $S^2(V)$ be the space of symmetric two tensors on $V$. In this paper, we always identify $V$ with its dual space $V^*$ via the metric $g$. Then 
 $S^2(V)$ splits into $O(V)$-irreducible subspaces as
\begin{equation*}
        S^2(V)=S^2_0(V)\oplus \R g,
\end{equation*}
where $S^2_0(V)$ is the space of traceless part of $S^2(V)$  and $g=\sum\limits_{i=1}^{n}e_i\otimes e_i$. We denote by $N$ the dimension of $S^2_0(V)$, that is 
\begin{equation*}
    N=\dim(S^2_0(V))=\tfrac{(n-1)(n+2)}{2}.
\end{equation*}
The space of symmetric two-tensors on $\wedge^2V$,  denoted by $S^2(\wedge^2 V)$,  has an orthogonal decomposition 
    \begin{equation*}
        S^2(\wedge^2 V) =S^2_B(\wedge^2 V) \oplus \wedge^4 V,
    \end{equation*}
where $S^2_B(\wedge^2 V)$ is the space of algebraic curvature operators on $V$, consisting of all tensors $R\in S^2(\wedge^2V)$  satisfying the first Bianchi identity. 
The tensor product $\otimes$ is defined by 
\begin{equation*}
        (e_i\otimes e_j)(e_k,e_l)=\delta_{ik}\delta_{jl},
    \end{equation*}
    where 
$$\delta_{ij}=\begin{cases}
    1,  & i=j, \\ 0, &i\neq j.
\end{cases}$$
For $e_i$ and $e_j$ in $V$, the symmetric product $\odot$ and wedge product $\wedge$ are defined by 
 \begin{equation*}
        e_i \odot e_j=e_i\otimes e_j +e_j \otimes e_i,
    \end{equation*}
    and 
\begin{equation*}
        e_i \wedge e_j=e_i\otimes e_j - e_j \otimes e_i,
    \end{equation*}
respectively. Furthermore the inner product on $\wedge^2(V)$ is defined  by 
    \begin{equation*}
        \langle A, B \rangle =\frac{1}{2}\tr(A^T B),
    \end{equation*}
and the inner product on $S^2(V)$ is defined by
    \begin{equation*}
        \langle A, B \rangle =\tr(A^T B).
    \end{equation*}
Clearly,  $\{e_i \wedge e_j\}_{1\leq i<j\leq n}$ forms an orthonormal basis of $\wedge^2(V)$ and $\{\dfrac{1}{\sqrt{2}}e_i \odot e_j\}_{1\leq i<j\leq n} \cup \{\dfrac{1}{2}e_i \odot e_i\}_{1\leq i\leq n}$ forms an orthonormal basis of $S^2(V)$. 

\subsection{Curvature Operator of the Second Kind}\label{sec 2.2}
For an algebraic curvature tensor $R\in S^2_B(\wedge^2(T_p M))$, $R$ can induce two self-adjoint operators (see \cite{Nishikawa86}). 
The first one denoted by $\hat{R}:\wedge^2 (T_p M) \to \wedge^2(T_p M)$ is   defined as 
    \begin{equation*}
        \hat{R}(\omega)_{ij}=\frac{1}{2}\sum_{k,l=1}^n R_{ijkl}\omega_{kl},
    \end{equation*}
which is called  the curvature operator of the first kind.
The other one,  $\overline{R}:S^2(T_p M) \to S^2(T_p M)$, is defined by
\begin{align*}
    \overline{R}(\vp)_{ij}=\sum_{k,l=1}^n R_{iklj}\vp_{kl}.
\end{align*}
The curvature operator of the second kind, introduced by Nishikawa \cite{Nishikawa86}, is 
symmetric bilinear form
\begin{align*}
    \mathring{R}: S^2_0(T_pM) \times S^2_0(T_pM) \to \R
\end{align*}
obtained by restricting $\overline{R}$ to $S^2_0(T_pM)$, the space of traceless symmetric two-tensors. It was pointed out in \cite{NPW22} that the curvature operator of the
second kind $\mathring{R}$ can also be interpreted as the self-adjoint operator
\begin{align}\label{2.1}
\mathring{R} = \pi \circ \overline{R}: S^2_0(T_pM) \to  S^2_0(T_pM),
\end{align}
where $\pi: S^2(T_pM) \to S^2_0(T_pM)$ is the projection map.

We denote by $\{\l_j\}_{j=1}^N$ the eigenvalues of $\mathring{R}$, 
by $
\bar{\l}=\sum_{j=1}^N \l_j/N
$ be the average of the eigenvalues of $\mathring{R}$,  by $R$ the  Riemann curvature tensor, and by $\S$ be the scalar curvature.
 For Einstein manifolds, one can obtain the scalar curvature via $\mathring{R}$. Precisely  
\begin{proposition}\label{prop 2.1}
Suppose $(M^n,g)$ is an Einstein manifold of dimension $n$.  Then the scalar curvature satisfies 
\begin{align}\label{2.2}
\S=n(n-1)\bar\lambda.   
\end{align}
\end{proposition}
\begin{proof}  
   Let $\{e_i\}_{i=1}^n$ be an orthonormal basis of $T_p M$, then
   $\{\dfrac{1}{\sqrt{2}}e_i \odot e_j\}_{1\leq i<j\leq n} \cup \{\dfrac{1}{2}e_i \odot e_i\}_{1\leq i\leq n}$ is an orthonormal basis of $S^2(T_pM)$. Since  
\begin{align*}
    \left\langle {\bar{R}(\frac{1}{{\sqrt 2 }}{e_i} \odot {e_j}),\frac{1}{{\sqrt 2 }}{e_i} \odot {e_j}} \right\rangle =R_{ijij},
\end{align*}
  then
\begin{align*}
       \operatorname{tr}(\bar{R})
 =& \sum\limits_{1\le i<j\le n} {\left\langle {\bar{R}(\frac{1}{{\sqrt 2 }}{e_i} \odot {e_j}),\frac{1}{{\sqrt 2 }}{e_i} \odot {e_j}} \right\rangle }  + \sum\limits_{1\le i\le n} {\left\langle {\bar{R}(\frac{1}{2}{e_i} \odot {e_i}),\frac{1}{2}{e_i} \odot {e_i}} \right\rangle }\\
 =& \sum\limits_{1\le i<j\le n} R_{ijij} \\
 =&\frac{\S}{2}.
\end{align*}
Thus, $$\operatorname{tr}(\mathring{R})=\operatorname{tr}(\bar{R})-\langle\bar{R}(\frac{g}{\sqrt n}), \frac{g}{\sqrt n}\rangle=\frac{\S}{2}+\frac{\S}{n}=\frac{n+2}{2n}\S.$$ 
Using $\operatorname{tr}(\mathring{R})=N \bar \lambda$, we obtain
$$\S=\frac{2n}{n+2}N \bar \lambda=n(n-1)\bar\lambda,$$
proving \eqref{2.2}.
\end{proof}

\begin{definition}[\cite{NPW22}]
\label{def2.4}
Let $\mathcal{T}^{(0, k)}(V)$ denote the space of $(0, k)$-tensor space on $V$. For $S\in S^{2}(V)$ and $T\in\mathcal{T}^{(0, k)}(V)$, we define 
\begin{align*}
    S:\  &\mathcal{T}^{(0, k)}(V) \to \mathcal{T}^{(0, k)}(V),\\
    &(ST)({X_1}, \cdots, {X_k}) = \sum\limits_{i = 1}^k {T({X_1}, \cdots , S{X_i}, \cdots, {X_k})},
\end{align*}
and define $T^{S^{2}}\in \mathcal{T}^{(0, k)}(V)\otimes S^{2}(V)$ by
\begin{align*}
    \left\langle
    T^{S^{2}}(X_{1}, \cdots, X_{k}), S
    \right\rangle
    =(ST)\left(X_{1}, \cdots,  X_{k}\right).
\end{align*}
\end{definition}
According to the above definition, if $\{\bar S^{j}\}_{j=1}^N$
 is an orthonormal basis for $S^{2}(V)$, then 
$$T^{S^{2}}=\sum\limits_{j=1}^N\bar S^{j}T\otimes\bar S^{j}.$$
Similarly, we define $T^{S^{2}_{0}}\in \mathcal{T}^{(0, k)}(V)\otimes S^{2}_{0}(V)$ by
$$T^{S^{2}_{0}}=\sum\limits_{j=1}^N S^{j}T\otimes S^{j},$$
where $\{ S^{j}\}_{j=1}^N$ is an orthonormal basis for $S^{2}_{0}(V)$.

\subsection{A formula for Weyl tensor on Einstein manifolds}
Given $A, B\in S^2(V)$, their Kulkarni-Nomizu product gives rise to $A \owedge B  \in  S^2_B(\wedge^2 V)$ via
\begin{equation*}
    (A \owedge B )_{ijkl} =A_{ik}B_{jl}+A_{jl}B_{ik} -A_{jk}B_{il}-A_{il}B_{jk}.
\end{equation*}
It is well known that the Riemann curvature tensor $R$ can be decomposed into irreducible components (cf. \cite[(1.79)]{CaM20}) as
\begin{align}\label{2.3}
    R=W+\frac{1}{n-2}\Ric\owedge g-\frac{\S}{2(n-1)(n-2)}g\owedge g,
\end{align}
where $W$ is the Weyl curvature tensor and $\Ric$ is the Ricci curvature tensor.   
In terms of any basis, we have
\begin{align*}
        R_{ijkl}=&W_{ijkl}+\frac{1}{n-2}\lf(R_{ik} g_{jl}+R_{jl} g_{ik}-R_{il} g_{jk}-R_{jk} g_{il}\ri)\\
        &-\frac{\S}{(n-1)(n-2)}\lf(g_{ik}g_{jl}-g_{il}g_{jk}\ri).
    \end{align*}
Moreover, if $(M^n, g)$ is an Einstein manifold, then $\Ric=\frac{\S}{n}g$, and  equation \eqref{2.3} simplifies to
\begin{align}\label{2.4}
    R=W+\frac{\S}{2n(n-1)}g\owedge g.
\end{align}

\section{Proof of Theorem \ref{thm1}}\label{sec 3}
The  proof of  Theorem \ref{thm1}  relies on establishing the key estimate
\begin{align*}
    \langle \Delta R, R\rangle\ge 0,
\end{align*}
where $R$ denotes the Riemann curvature tensor.  Our approach begins by expressing $\langle \Delta R, R\rangle$ in terms of the curvature operator of the second kind  under the cone condition \eqref{1.1} for Einstein manifolds. 

We recall  some curvature notations on $(M^n, g)$:
 $R$ represents  the Riemann curvature operator,   $W$ represents  the Weyl tensor, $\mathring{R}$ represents  the curvature operator of the second kind,   $\l_1\le\l_2\le \cdots\le \l_N$ are the eigenvalues of  $\mathring{R}$, and  $\bar\l=\sum_{j=1}^N \l_j/N$ is the average of $\l_j$.
To prove the main theorem, we conduct several computations involving the Riemann curvature through the curvature operator of the second kind on Einstein manifolds. 
\begin{lemma}\label{lm 3.1}
Let $(M^n,g)$ be an $n$-dimensional Einstein manifold. Then 
\begin{align}\label{3.1}
    |R|^2=|W|^2+2n(n-1)\bar\lambda^2,
\end{align}
and
\begin{align}\label{3.2}
    \sum_{j=1}^N \lambda_j^2
    = \frac{3}{4} \lf| R \ri|^2 - (n-1)^2 \bar{\lambda}^2.
\end{align}
\end{lemma}
\begin{proof}
Recall from the equation \eqref{2.4} that 
\begin{align*}
    R= \frac{\S}{2n(n-1)}g \owedge g +W,
\end{align*}
then 
\begin{align*}
    |R|^2
    = \frac{\S^2}{4n^2(n-1)^2}|g \owedge g|^2 +|W|^2 =\frac{\S^2}{2n(n-1)} +|W|^2,
\end{align*}
where we used the identity $|g \owedge g|^2=8n(n-1)$ (cf.
 \cite[Lemma 7.22]{Lee18}). 
Combining this with identity \eqref{2.2}, we conclude
\begin{align*}
    |R|^2=|W|^2+2n(n-1)\bar\lambda^2,
\end{align*}
giving \eqref{3.1}.

Recall from Lemma 3.3 of \cite{DF24}  that
\begin{align*}
    \sum_{j=1}^N \lambda_j^2
    = \frac{3}{4} \lf| R \ri|^2 - \frac{\S^2}{n^2},
\end{align*}
and use identity \eqref{2.2} again, we have
\begin{align*}
    \sum_{j=1}^N \lambda_j^2
    = \frac{3}{4} \lf| R \ri|^2 - (n-1)^2\bar{\l}^2,
\end{align*}
giving \eqref{3.2}.
\end{proof}
\begin{lemma}
\label{lemma 2}
Let $(M^n,g)$ be an Einstein manifold of dimension $n$.  Suppose the smallest eigenvalue of $\mathring{R}$ satisfies $\lambda_1 \ge -\theta \bar \lambda$, then
\begin{align}\label{3.3}
    \sum_{j=1}^N
    \lambda_{j} |S^{j} W|^2 
    \ge -\frac{8(n^2+n-8)}{3n}\theta \bar\lambda \sum_{j=1}^N \lambda_j^2 
    +\frac{4(n^2+n-8)(n-1)(n+2)}{3n} \theta \bar\lambda^3,
\end{align}
where   $S^{j} W$ is given by Definition \ref{def2.4}.
\end{lemma}

\begin{proof}
Recall from  equality (4.1) of \cite{DF24} that
\begin{align*}
 \sum_{j=1}^N|S^j W|^2 =\frac{2(n^2+n-8)}{n}|W|^2,
\end{align*}
and by  assumption $\lambda_1 \ge -\theta \bar \lambda$, we estimate
that
\begin{align}\label{3.4}
    \sum_{j=1}^N
    \lambda_{j} |S^{j} W|^2 
    \ge- \theta \bar \lambda \sum_{j=1}^N |S^j W|^2
    =- \theta \bar \lambda \frac{2(n^2+n-8)}{n}|W|^2.
\end{align}
From \eqref{3.1} and \eqref{3.2}, we have 
\begin{align*}
    |W|^2
    = \frac{4}{3} \sum_{j=1}^N \lambda_j^2
    -\frac{2}{3}(n-1)(n+2)\bar\lambda^2,
\end{align*}
and substituting above equality into \eqref{3.4}, we conclude inequality  \eqref{3.3} holds.
\end{proof}

\begin{lemma}\label{lm 3.3}
Let $(M^n,g)$ be  an $n$-dimensional Einstein manifold.
Suppose the curvature operator of the second kind satisfies $\lambda_1 \ge -\theta \bar \lambda$. Then
\begin{align}\label{3.5}
\begin{split}
    \frac{9n}{16} \langle \Delta R, R \rangle
    \ge & \Big[ N(N-3)\theta-(2N-9n+6)N \Big]\bar\l^3\\
   &+  \Big[ (2N-12n+6)-(N-3)\theta \Big] \bar\lambda \sum_{j=1} ^{N}\lambda_{j}^2+3n \sum_{j=1}^{N}\lambda_{j}^3,
    \end{split}
\end{align}
where $\Delta$ is the Laplace-Beltrami operator with respect to the metric $g$. 
\end{lemma}

\begin{proof}
Recall from equality (3.3) of \cite{DF24}  that
\begin{align*}
    3\langle \Delta R,R\rangle
  =&\displaystyle\sum_{j=1}^N {{\lambda _j }{{\left| {{S^j }W} \right|}^2}}+8\left(\frac{-n^3+6n^2+12n-8}{3n^4(n-1)^2} \right){\S^3}\\
   &+8\left(\frac{2n^2-22n+8}{3n^2(n-1)} \right)\S\displaystyle \sum_{j=1}^N  {\lambda _j ^2}  + 16\displaystyle \sum_{j=1}^N {\lambda_j ^3},
\end{align*}
and using $\S=n(n-1)\bar \lambda$, we obtain
\begin{align*}
    3\langle \Delta R,R\rangle
  =&\sum_{j=1}^N {{\lambda _j }{{\left| {{S^j }W} \right|}^2}}
  +\frac{8(n-1)}{3n}\lf(-n^3+6n^2+12n-8\ri){\bar\lambda^3}\\
   &+8\lf(\frac{2n^2-22n+8}{3n} \ri) \bar\lambda  \sum_{j=1}^N \lambda _j ^2  + 16\sum_{j=1}^N {\lambda _j ^3}.
\end{align*}
By  \eqref{3.3} and $N=(n-1)(n+2)/2$, we derive  that
\begin{align*}
    3\langle \Delta R,R\rangle
    \ge & 
    -\frac{8(n^2+n-8)\theta}{3n} \bar\lambda \sum_{j=1}^N  \lambda_j^2 +\frac{4(n^2+n-8)(n-1)(n+2)}{3n} \theta \bar\lambda^3+ 16\sum_{j=1}^N {\lambda _j ^3}\\
    &+\frac{8(n-1)}{3n}
    \lf(-n^3+6n^2+12n-8 \ri){\bar\lambda^3}
    +8\lf(\frac{2n^2-22n+8}{3n} \ri) \bar\lambda \sum_{j=1}^N  \lambda _j ^2 \\
    =\quad& \frac{16}{3n}\Big[ N(N-3)\theta-(2N-9n+6)N \Big]\bar\l^3\\
   &+ \frac{16}{3n} \Big[ (2N-12n+6)-(N-3)\theta \Big] \bar\lambda \sum_{j=1} ^{N}\lambda_{j}^2+16 \sum_{j=1}^{N}\lambda_{j}^3,
\end{align*}
proving \eqref{3.5}. 
\end{proof}

Let $\l:=(\l_1, \l_2, \cdots, \l_N)$, and we denote  the right hand side of \eqref{3.5} by $f(\l)$, namely   
\begin{align*}
\begin{split}
    f(\l):=&\Big[ N(N-3)\theta-(2N-9n+6)N \Big]\bar\l^3\\
   &+  \Big[ (2N-12n+6)-(N-3)\theta \Big] \bar\lambda \sum_{j=1} ^{N}\lambda_{j}^2+3n \sum_{j=1}^{N}\lambda_{j}^3.
\end{split}
\end{align*}
\begin{proposition}\label{pro 3.4}
Let $(M^n,g)$ be an $n$-dimensional Einstein manifold. Suppose the curvature operator of the second kind satisfies \eqref{1.1}. Then 
\begin{align}\label{k-e}
   \left\langle \Delta R, R\right\rangle\ge 0,
\end{align}
with the equality holding if and only if 
\begin{align*}
    \l=(1,1,\cdots, 1 )\bar \l \text{\quad or \quad }  \l=(-\theta,\frac{N+\theta}{N-1},\cdots,\frac{N+\theta}{N-1})\bar{\l}.
\end{align*}
\end{proposition}
\begin{proof}
By the curvature condition \eqref{1.2}, we have $\bar \l\ge 0$. If $\bar \l=0$, then $\l_1=\l_2=\cdots=\l_N=0$, and \eqref{k-e} holds trivially by \eqref{3.5}. Thus, we  focus on the case $\bar \l>0$.

Let 
$\Tilde{\lambda}_{j}=\l_{j}/\bar{\l}$,   $\tilde{\l}=(\tilde{\l}_1, \tilde{\l}_2, \cdots, \tilde{\l}_N)$, and then $\tilde{\l}$ satisfies the constraints
\begin{align}\label{3.6}
\sum_{j=1}^N \tilde{\l}_j=N, \quad \text{ and }\quad  \tilde{\l}_j \ge -\theta.  
\end{align}
Set $x=(\tilde{\l}_1+\theta, \tilde{\l}_2+\theta, \cdots, \tilde{\l}_N+\theta)$ and $t=1+\theta$, then the constraints \eqref{3.6} are equivalent to 
\begin{align*}
   \sum_{j=1}^{N}x_{j}=Nt \quad \text{and}\quad x_j\ge 0.
\end{align*}

We prove the case of $n\ge 6$. From the expression of $f(\l)$ we can deduce that 
\begin{align*}
    \frac{1}{\bar\l^3} f(\l)=C( n, t)+3n F(x),
\end{align*}
where 
\begin{align*}
   C(n, t)=&N\big[ (N-3)t-(3N-9n+3) \big] +3nN(t-1)^2(2t+1) \\
    &+ \big[3(N-4n+1)-(N-3)t\big]N(1-t^2),
\end{align*}
and
\begin{align*}
F(x)=-(3-\frac{N-2}{N-1})t\sum_{j=1}^N x_{j}^2+\sum_{j=1}^N x_{j}^3.
\end{align*}

Now we  reduce the proposition  to seek the minimal points of 
$F(x)$ with constraints 
\begin{align*}
    \sum_{j=1}^N x_j=Nt, \quad \text{and}\quad  x_j\ge 0 \ (j=1,2,\cdots, N).
\end{align*}
The corresponding Lagrangian function is 
$$
\mathcal{L}(x, \mu):=-(3-\frac{N-2}{N-1})t\sum_{j=1}^N x_{j}^2+\sum_{j=1}^Nx_{j}^3
    +\mu(\sum_{j=1}^N x_{j} -Nt).
$$
Let
$$
\Omega_N=\Big\{x\in \R^N:  \sum_{j=1}^N x_j=Nt, x_j>0\Big \},
$$
and 
$$
\Omega_{k}=\Big\{x\in \R^N: \sum_{j=1}^N x_j=Nt,  x_1=\cdots=x_{N-k}=0, x_{j}>0, j\ge N-k+1 \Big\}
$$
for $k=1, \cdots, N-1$.

In $\Omega_k$, using the method of Lagrange multipliers, we find that the critical points $$x=\big( \underbrace{0, \cdots ,0}_{N-k}, x_{N-k+1}, \cdots, x_N 
    \big)$$ of $\mathcal{L}(x,\mu)$  satisfy 
\begin{align*}
   \begin{cases}
       3 x_i^2-2(3-\frac{N-2}{N-1}) t x_i + \mu=0,\\
       \sum_{i=N-k+1}^N x_i=Nt,
   \end{cases} 
\end{align*}
hence the possible critical points are given by
\begin{align*}
  P_{k, l}:=  \big( {
    \underbrace 
    {0, \cdots ,0}_{N-k},
    \underbrace 
    {a, \cdots ,a}_{l},
    \underbrace
    {b, \cdots ,b}_{k - l}
    } \big),
\end{align*}
for $0\le l\le \frac{k} 2$, where $a$ and $b$ are defined through
\begin{align}\label{3.7}
\begin{cases}
a+b=2(1-\frac{N-2}{3(N-1)})t,\\
la+(k-l)b=Nt.
\end{cases}
\end{align}

When $l$ equals $k/2$, we observe that $1-\frac{N-2}{3(N-1)}<\frac N k$, and in this case equation \eqref{3.7} becomes unsolvable.
When $l<\frac k 2$, by solving \eqref{3.7} we get
\begin{align*}
    a=(A-B)t, \quad b=(A+B)t,
\end{align*}
where $A=1-\frac{N-2}{3(N-1)}$, $B=\frac{N-Ak}{k-2l}$, and
direct calculation gives
\begin{align*}
   F(P_{k,l})
    =t^3
    \lf(-2kA^3-3A^2(N-Ak)+\frac{(N-Ak)^3}{(k-2l)^2}
    \ri),
\end{align*}
hence 
\begin{align}\label{3.9}
F(P_{k, 0})<F(P_{k, l}), \quad 1\le l<k/2.
\end{align}
 Therefore we conclude from  \eqref{3.9} that
\begin{align*}
    \min\left\{F(x):  \sum_{j=1}^N x_j=N t, x_j\ge 0\right\}=\min\left\{F(P_{k, 0}), k=1,2,\cdots, N\right\},
\end{align*}
which indicates that the minimum points of the function $f(\l)$, subject to the constraint  $\l_i\ge -\theta \bar \l$, are included among
\begin{align*}
 \l^m=\left(
    \underbrace{-\theta\bar{\l},\cdots,-\theta\bar{\l}}_{m},
    \underbrace
    {\frac{N+m\theta}{N-m}\bar{\l},\cdots, \frac{N+m\theta}{N-m}\bar{\l}}_{N-m}
    \right)  
\end{align*}
for $m=0, 1, \cdots, N-1$. 

For $\l=\l^m$, we calculate that 
\begin{align*}
    \frac{1}{\bar\l^2}\sum_{j=1}^{N} \lambda^2_{j}
   =m \theta^2+(N-m)\lf(\frac{N+m \theta}{N-m}\ri)^2
   =N+\frac{N}{N-m}mt^2,
   \end{align*}
and 
\begin{align*}
     \frac{1}{\bar\l^3}\sum_{j=1}^{N} \lambda^3_{j}
   =-m t^3 \frac{(N-2m)N}{(N-m)^2}+\frac{3mN}{N-m}t^2+N,
\end{align*}
we then conclude  
\begin{align}\label{3.10}
    \frac{1}{\bar\l^3}f(\l^m)=
    &\Big[-\frac{N(N-3)m}{N-m}-\frac{3nmN(N-2m)}{(N-m)^2}\Big]t^3+\frac{3N(N+1-n)}{N-m}mt^2 \\
    =& \frac{mNt^2}{N-m}
    \Big[3(N+1-n)-(N-3+\frac{3n(N-2m)}{N-m})t\Big].\nonumber
\end{align}
Recall from the definition of $\theta$ in \eqref{1.2} that
$$
\theta=\frac{2 \left(n^4-4 n^3-5 n^2+30 n-16\right)}{n^4+8 n^3-5 n^2-48 n+32}=\frac{3(N-1)(N+1-n)}{(N-1)(N-3)+3n(N-2)}-1,
$$
we see 
$$
t=1+\theta=\frac{3(N+1-n)}{N-3+3n(N-2)/(N-1)},
$$
and substituting above to \eqref{3.10} yields
$$
    f(\l^0)=f(\l^1)=0, \quad \text{and}\quad 
f(\l^k)>0$$
 for $k=2,3, \cdots, N-1$.
 Hence Proposition \ref{pro 3.4} comes true for $n\ge 6$.

For dimensions $n=4$ and $5$,  $\langle \Delta R, R\rangle$ 
can be expressed explicitly in terms of the curvature operator of the second kind as (see  \cite[Section 4.2]{DF24}):
\begin{align}\label{3.22}
     \left\langle \Delta R, R\right\rangle=\begin{cases}
      8\left(\sum_{j=1}^9 \l_j^3-9\bar\l^3\right), &\quad n=4,\\ 
      8\left(\sum_{j=1}^{14} \l_j^3+\frac{1}{3}\bar\l\sum_{j=1}^{14}\l_j^2-\frac{56}{3}\bar\l^3\right), &\quad n=5.
     \end{cases}
\end{align}
Then an argument similar to the one used for $n\ge 6$ shows that Proposition   \ref{pro 3.4}  holds for $n=4$ and $n=5$.  This completes the proof of Proposition   \ref{pro 3.4}. 

\end{proof}
Now we prove the main theorem. 

\begin{proof}[Proof of Theorem \ref{thm1}]
Using \eqref{k-e}, we have
\begin{align*}
     \Delta |R|^2
    =2|\nabla R|^2 + 2\langle \Delta R,R \rangle \ge 2|\nabla R|^2 ,
\end{align*}
and integrating the above inequality over $M$ yields
\begin{align*}
    0=\int_M  \Delta |R|^2\, d\mu_g \ge\int_M 2|\nabla R|^2 \, d\mu_g \ge 0,
\end{align*}
which implies $|\nabla R|=0$. Therefore $M$ is a symmetric space. \cite[Theorem B]{NPW22} and \cite[Theorem 1.8(2)]{Li24} imply that $M$ is either flat or a rational homology sphere. 
Besides, Poincar\'e duality shows that rational cohomology is determined by rational homology \cite[Chapter 3.3]{Hatcher02}. 
Thus we conclude that $M$ is either flat or a symmetric space which is also a rational cohomology sphere. 
Compact symmetric spaces which are real cohomology spheres were classified by Wolf in \cite[Theorem 1]{Wolf}. 
If $\bar{\l}>0$, apart from spheres,  $SU(3)/SO(3)$ is the only simply connected example. 
Proposition \ref{pro 3.4} implies that the eigenvalues of $\mathring{R}$ is either $\lambda=(1,\cdots,1)\bar{\l}$ or $\lambda=(-\theta, \frac{N+\theta}{N-1},\cdots, \frac{N+\theta}{N-1})\bar{\l}$, and 
combining with \cite[Example 4.5]{NPW22}, we know that $M$ is a round sphere.
\end{proof}

\section*{Acknowledgments}
The first author expresses her gratitude to her advisor, Professor Ying Zhang, for lots of encouragement and helpful suggestions. The authors would like to thank Professor Xiaolong Li for many helpful discussions. The research of this paper is partially supported by NSFC Grant No.12171345 and NSF of Jiangsu Province Grant No.BK20231309.

\bibliographystyle{plain}
\bibliography{ref}
\end{document}